\documentclass[letterpaper, 10 pt, conference]{ieeeconf}  

\IEEEoverridecommandlockouts                              
\usepackage[utf8]{inputenc} 
\usepackage[T1]{fontenc}    
\usepackage{hyperref}       
\usepackage{amssymb,amsfonts,amsmath,mathrsfs,bm,bbm}
\usepackage{graphicx}
\usepackage{gensymb}
\usepackage{epstopdf}
\usepackage{float}
\usepackage{graphicx}
\usepackage{epsfig}
\usepackage{dsfont}
\usepackage{color}
\usepackage{xcolor}
\usepackage{url}
\usepackage{cite}
\usepackage{booktabs}       
\usepackage{nicefrac}       
\usepackage{microtype}      

\newcommand{\mul}{\underline{\mu}}
\newcommand{\muu}{\overline{\mu}}
\newcommand{\ql}{\underline{q}}
\newcommand{\qu}{\overline{q}}
\newcommand{\R}{\mathbbm{R}}
\newcommand{\rev}{\mathcal{R}}
\newcommand{\s}{\mathcal{S}}
\newcommand{\m}{\mathcal{M}}
\newcommand{\Lag}{\mathcal{L}}

\newtheorem{theorem}{Theorem}
\newtheorem{assumption}{Assumption}
\newtheorem{condition}{Condition}

\newtheorem{lemma}{Lemma}

\title{\LARGE \bf
General Revenue Adequacy Conditions for Energy Transport Networks
} 
\author{Sidhant Misra$^1$, Marc Vuffray$^1$, Anatoly Zlotnik$^1$, and Aleksandr M. Rudkevich$^2$
\thanks{This project was supported by the LDRD program and the Center for Nonlinear Studies at Los Alamos National Laboratory, as well as the Advanced Grid Modeling Research program of the U.S. DOE Office of Electricity. Research conducted at Los Alamos National Laboratory is done under the auspices of the National Nuclear Security Administration of the U.S. Department of Energy under Contract No. 89233218CNA000001. Report No. LA-UR-26-20175.}
\thanks{$^1$\{sidhant, vuffray, azlotnik\}@lanl.gov, \,\,  Theoretical Division, Los Alamos National Laboratory, Los Alamos, NM, USA 87545}
\thanks{$^2$arudkevich@negll.com, \,\, Newton Energy Group, Boston, MA, USA 02116}}

\begin{document}

\maketitle
\thispagestyle{empty}
\pagestyle{empty}

\begin{abstract}
    Optimization is widely used to determine the physical and financial exchange of wholesale electricity in organized markets.  Guarantees of solution optimality and feasibility rest largely on convexity, which is not in general a characteristic of the governing equations for  power grid and gas pipeline networks.  Policy decisions that base the scheduling and locational pricing of electricity transactions on optimization rely on the guarantee of revenue adequacy, which ensures that the market administrator will collect enough payments in congestion rents to settle financial transmission rights.  Developing a similar mechanism for locational trade valuation of natural gas also requires assurance that pricing outcomes are revenue adequate, and also cover the costs of gas compressor operation.  However, it has been shown that the AC power flow equations are in general non-convex and hence conditions for guaranteeing revenue adequacy in optimal power flow solutions are challenging to generalize.   In this study, we develop a general formal mathematical setting for nonlinear physical network flows and examine the conditions for revenue adequacy.  The result is verified for DC and AC power flow as well as steady-state gas flow in a pipeline network.
\end{abstract}

\section{Introduction} \label{sec:intro}

A critical issue in commodity market design, particularly involving transportation and/or networks, is revenue adequacy \cite{macher2014revenue}.  Revenue adequacy was first examined in the context of railroad transport, where it refers to a higher return on investment than the cost of capital.   Substantial analysis of revenue adequacy has been developed for electricity transmission capacity markets \cite{harvey1996transmission}, where the term refers to the requirement for a market mechanism to guarantee an operating surplus under forseeable conditions, and has been used to justify major policy decisions \cite{oneill2012joint}.  However, many established theories on market design typically assume convex problems with linear constraints, which leads to inconsistencies in practice.

Mechanisms that utilize locational spot pricing, capacity rights, and congestion rents are extensively studied for the design of electricity markets \cite{schweppe2013spot}.  Spot pricing was developed as a mechanism to improve the efficiency of power transmission as part of a larger framework called homeostatic control, which includes longer-term contracts \cite{caramanis1982optimal}. The policies of independent system operators (ISOs), which are organizations set up to operate regulated electricity markets, typically are designed to ensure that enough congestion rents are collected to pay financial transmission rights (FTRs).  The ISOs clear the wholesale market by solving a series of optimal power flow (OPF) problems as mathematical programs for real-time and forward physical scheduling and pricing, and participants are subject to complex rules on when and for whom outcomes are operationally and financially binding \cite{sarkar2008comprehensive}.  The mechanisms in use today are based on DC OPF modeling \cite{harvey1996transmission}, and while locational marginal price (LMP) computations have been extended to AC power flow \cite{sarkar2011optimal}, obtaining a feasible result that is also revenue adequate in FTR settlement may require constraint relaxations and/or removal of problematic graph paths.  Although it is claimed that instances requiring such adjustments are very rare, this would depend on many factors including network topology and market conditions \cite{sarkar2011optimal}. 

Interestingly, it was proven that accounting for AC power flows in electricity markets, with nonlinear, non-convex, continuous OPF, revenue adequacy cannot be guaranteed even for simple networks \cite{lesieutre2005convexity}.  Though power injection at loads, generator active and reactive power supply, voltage magnitudes, and line flow capacities are assumed to be convex in the OPF, the cost function may be non-convex.  Specifically, detailed models of thermal generators may have non-convex and discontinuous operating costs, and researchers continue to investigate non-convexities and techniques to handle them \cite{narimani2018empirical}.  Recent studies examined revenue adequacy in shifting from linear to conic OPF frameworks \cite{ratha2023moving}, and emerging prosumer markets \cite{baroche2019prosumer}.

Auction-based pricing was also examined for pipeline transport of natural gas. A linear programming model was proposed for auctioning transportation rights \cite{united1987gas}, and was extended to nonlinear flow  \cite{rudkevich2017hicss}. There, revenue adequacy was shown for a general class of optimal gas flow (OGF) problems provided the network is adequately instrumented such that zero flow is always feasible \cite{rudkevich2017hicss}.  That result is based on properties of gas flow over networks under the usual physical operating regime, with monotonicity in solutions with respect to parameters \cite{vuffray2015monotonicity}.

In this study, we develop a formal setting for analysis of optimization-based market mechanisms for allocation and pricing of a commodity transported through a nonlinear physical flow network 
and examine the general conditions for revenue adequacy. Our approach is based on a decomposition of the decision variables in an optimal flow problem into those associated with traded quantities bounded and valued according to bids by participants, and the dependent physical variables that describe potentials and flows throughout the network.  The analysis of revenue adequacy depends on certain properties of the feasible set of solutions, as well as the Karush-Kuhn-Tucker (KKT) conditions for optimality.  

The manuscript proceeds as follows:  In Section \ref{sec:model}, we present a general physical flow allocation model and conditions on revenue adequacy.  Section \ref{sec:global_star_sets} shows global revenue adequacy results for DC OPF and steady-state gas flow problems. Section \ref{sec:acopf} shows a local property for AC OPF, and we conclude with policy implications in Section \ref{sec:conc}.

\section{General Model and Revenue Adequacy Result} \label{sec:model}

We consider an abstract market for trading quantities of a commodity exchanged among $N$ participants according to relations described by additional system state variables.  We suppose that the system state and control variables are also constrained by inequalities.  The constraints that specify conservation of the commodity entering or leaving the market and bounds on traded quantities are of particular interest and are identified separately from the remaining relations.

\vspace{-1ex}
\subsection{General Commodity Market Model}  \label{sec:flow_network}

Suppose that a general collection of relations represents a system used to realize a commodity exchange outcome among participants. The governing laws are specified by equations  
\begin{align} \label{eq:h_fun}
    h_j(x,y) = 0, \quad j=1,\ldots,N_h.
\end{align}
Here, we suppose that $x \in \R^{N}$ is a vector of $N$ system variables that are exposed to the market in the sense that they depend directly on quantities of the commodity being traded, i.e., supplied or consumed, by the $N$ users of the system. The variables $y \in \R^{N_y}$ denote quantities that are dependent on $x$ as well as those quantities that are controllable by the administrator of the system. The set of inequality constraints
\begin{align} \label{eq:g_fun}
    g_i(x,y) \leq 0, \quad i=1,\ldots,N_g
\end{align}
are used to represent limitations on the ability of the system to realize the commodity exchange, which in practice could depend on physical, engineering, legal, or other qualifications that affect the capacity in space and time. 
For compactness of notation, we introduce the vector valued functions 
\begin{subequations}
\begin{align} 
    H(x,y) = \{h_i(x,y), \ i=1,\ldots, N_h\}, \label{eq:H_fun_def} \\
    G(x,y) = \{g_i(x,y), \ i=1,\ldots, N_g\}, \label{eq:G_fun_def}
\end{align}
\end{subequations}
and denote the feasible set of the system by
\begin{align} \label{eq:feas_set}
    \s = \{(x,y) \mid H(x,y)=0, \ G(x,y) \leq 0\}.
\end{align}

Now let $q \in \R^N$ denote the vector of quantities of the commodity supplied to the system (or consumed, if negative) corresponding to each of the $N$ market participants.  We suppose that the quantities $q$ are to be optimized subject to minimum and maximum bound values $\underline{q}$ and $\overline{q}$, which are called \emph{nominations}.  The overall receipts by the system are given by the linear conservation law
\begin{align}  \label{eq:cons_law}
    x = q - d,
\end{align}
where $d \in \mathbb{R}^N$ denotes commodity outflows that are fixed \emph{a priori} and not determined by the market. We can then define the allowable set of market outcomes by
\begin{align} \label{eq:market_set}
    \mathcal{M} = \{(q,x) \mid x = q - d, \ \underline{q} \leq q \leq \overline{q}\}.
\end{align}
The problem of minimizing the cost $c(q)$ operating the system can be represented by the following optimization problem:
\begin{subequations}  \label{eq:min_cost_formulation}
\begin{align}
    \min_{q,x,y} \quad & c(q) & & \mbox{economic cost,}     \label{eq:economic_cost} \\
    \mbox{s.t.} \quad & \underline{q} \leq q \leq \overline{q} & \mul,\muu \qquad &\mbox{nomination limits,} \label{eq:q_bounds} \\
    & x = q - d & \lambda \qquad & \mbox{conservation law,} \label{eq:conservation_law} \\
    & H(x,y) = 0 & \nu_e \qquad & \mbox{system physics,} \label{eq:system_physics} \\
    & G(x,y) \leq 0. & \nu_i \qquad &\mbox{state \& control limits.}  \label{eq:engineering_limits}
\end{align}
\end{subequations}
The second column above lists the dual variables corresponding to each of the constraints. In some settings, we may seek to maximize the economic value or surplus produced for system users, in which case we express the economic cost to be minimized as the negative of the economic surplus, i.e., $c(q)=-s(q)$. The above optimization problem is in general non-convex and it may not be possible to efficiently solve the problem to global optimality.  The problem \eqref{eq:min_cost_formulation} can be expressed using the feasible system and market sets as
\begin{subequations}  \label{eq:min_cost_set_formulation}
\begin{align}
    \min_{q,x,y} \quad & c(q) && \mbox{economic cost,}     \label{eq:economic_cost_set} \\
    \mbox{s.t.} \quad & (q,x)\in\m, &&\mbox{allowable market outcomes,}\label{eq:m_feasibility} \\
    &(x,y)\in\s,  &&\mbox{system feasibility.}  \label{eq:s_feasibility}
\end{align}
\end{subequations}
With the formulation \eqref{eq:min_cost_set_formulation}, constraints specifying physical energy transport over the system are separated from constraints related to market participants that supply and consume energy, and these sets are decoupled in the subsequent analysis.

\vspace{-1ex}
\subsection{Revenue Adequacy}   \label{sec:rev_ad_general}

A local optimum for the extremal problem in formulation \eqref{eq:min_cost_formulation} can correspond to a collection of nominations for clearing a single or double sided market in a single auction. Let $(q,x,y)$ and $(\mul,\muu,\lambda,\nu_i,\nu_e)$ denote the primal-dual variables corresponding to such a local optimum. We investigate a particular pricing strategy, where the dual variables $\lambda$ of the conservation equations in \eqref{eq:conservation_law} are used to price the commodity $q$. The net revenue available to the system operator under such a pricing scheme is given by the sum of prices times inflows, \vspace{-2ex}
\begin{align}   \label{eq:revenue}
    \rev = \lambda^T(d-q) = -\lambda^T x. 
\end{align}
Any viable pricing strategy requires that net revenue is non-negative, i.e. the value of delivery is greater than cost of supply, and this condition, denoted by $\rev \geq 0$, is known as \emph{revenue adequacy}. We derive sufficient conditions for a market mechanism for commodity allocation and pricing stated in the form of problem \eqref{eq:min_cost_formulation} to be revenue adequate. 

Our first assumption ensures that the primal-dual solution pair $(q,x,y)$ and $(\mul,\muu,\lambda,\nu_i,\nu_e)$ that results from local extremization is regular. This can be enforced by several of the so-called constraint qualification conditions.  
\begin{condition}  \label{con:cq_condition}
The primal-dual set of variables satisfy the Mangasarian-Fromovitz constraint qualification (MFCQ) \cite{bertsekas1997nonlinear}. The gradients of the equality constraints are linearly independent at $(q,x,y)$, and there exists a vector $\delta\in\mathbb{R}^{N+N_y}$ such that $\nabla g_i(x,y)^T \delta < 0$ for all inequality constraints $i=1,\ldots,N_g$, and $\nabla h_j(x,y)^T \delta = 0$ for all equality constraints $j=1,\ldots,N_h$.
\end{condition}

Let $\s_x = \{x \mid (x,y) \in \s\}$ denote the projection of the feasible set $\s$ on the $x$ variables. Our second condition imposes a geometrical property on the projected set $\s_x$.
\begin{condition}  \label{con:local_condition}
The point $(x,y) \in \s$ is said to satisfy the locally star shaped (LSS) property if there exists $\epsilon\in(0,1)$ such that  $s x \in \s_x$ for all $s \in [1-\epsilon, 1]$.
\end{condition}
As will be shown later, Condition~\ref{con:local_condition} is sufficient for establishing revenue adequacy. The following is a stronger condition, which is useful to consider because it is satisfied by several prominent systems including linearized power grid networks and gas pipeline networks modeled using steady-state physics. 
\begin{condition}   \label{con:global_condition}
The set $\s_x$ satisfies the globally star shaped (GSS) property if for every $x \in \s_x$, $\{s x \mid s \in [0,1]\} \subseteq \s_x$.
\end{condition}
The following lemma follows directly from Condition \ref{con:global_condition}.
\begin{lemma}   \label{eq:condition2_to_1}
If $\s_x$ satisfies the GSS property, then every point $(x,y) \in \s$ satisfies the LSS property.
\end{lemma}

\noindent The main result of this paper is given in the following theorem.
\begin{theorem}[Revenue Adequacy]   \label{thm:revenue_adequacy_general}
Suppose that $(q,x,y)$ and $(\mul,\muu,\lambda,\nu_i,\nu_e)$ denote a primal-dual set of local optimal variables for \eqref{eq:min_cost_formulation} and assume that Condition~\ref{con:cq_condition} holds.  If  Condition~\ref{con:global_condition} holds, then the pricing scheme specified by $\lambda$ satisfies $\rev \geq 0$.  If only Condition~\ref{con:local_condition} holds, then $\lambda$ satisfies $\rev \geq 0$ also.
\end{theorem}
\begin{proof} 
The Lagrangian of problem \eqref{eq:min_cost_formulation} is given by
\begin{align} 
    \Lag & = c(q) +  \mul^T (\ql - q)   + \muu^T(q - \qu) + \lambda^T(x-q+d) \nonumber \\ & \qquad + \nu_i^T G(x,y) + \nu_e^T H(x,y). \label{eq:thm1_lagrangian}
\end{align}

\noindent As a part of the KKT conditions at the optimum, we get
\begin{subequations}
\begin{align}
    \nabla_x \Lag &= \lambda^T + \nu_i^T \nabla_x G + \nu_e^T \nabla_x H = 0    \label{eq:partial_x} \\
    \nabla_y \Lag &= \nu_i^T \nabla_y G + \nu_e^T \nabla_y H = 0.   \label{eq:partial_y}                
\end{align}
\end{subequations}

\noindent Let $\mathcal{F}$ be the set of all feasible directions $(\delta_x, \delta_y)$ at the current point, and let $\mathcal{F}_x = \{\delta_x \mid \exists (\delta_x,\delta_y) \in \mathcal{F}\}$ denote its $x$-projection.
\noindent A feasible direction $(\delta_x, \delta_y) \in \mathcal{F}$  must satisfy 
\begin{subequations}
\begin{align}
    \nu_i^T \left(\nabla_x G \delta_x + \nabla_yG \delta_y\right) &\leq 0,  \label{eq:inequality_feasible_ray} \\
    \nabla_x H \delta_x + \nabla_y H \delta_y &= 0. \label{eq:equality_feasible_ray}
\end{align}
\end{subequations}

\noindent If Condition~\ref{con:global_condition} holds (GSS), we must have $\delta_x  = -x \in \mathcal{F}_x$. Therefore there exists $\delta_y$ such that $(\delta_x,\delta_x) = (-x,\delta_y) \in \mathcal{F}$. Using this fact, we bound the total revenue as
\begin{align}
    \rev = -\lambda^T x &\stackrel{(a)}{=} \left(\nu_i^T \nabla_x G + \nu_e^T \nabla_x H\right)x \nonumber\\
    &\stackrel{(b)}{=} -\left(\nu_i^T \nabla_x G + \nu_e^T \nabla_x H\right)\delta_x \nonumber\\
    &\stackrel{(c)}{=} -\nu_i^T \nabla_x G \delta_x + \nu_e^T \nabla_y H \delta_y \nonumber\\
    &\stackrel{(d)}{=} -\nu_i^T \nabla_x G \delta_x  - \nu_i^T \nabla_y G \delta_y \geq 0, \label{eq:thm1_rev_ad}
\end{align}
where $(a)$ follows from \eqref{eq:partial_x}, $(b)$ holds because $\delta_x = -x$, $(c)$ follows by using \eqref{eq:equality_feasible_ray} and $(d)$ follows from \eqref{eq:partial_y}.  The inequality follows from the MFCQ in Condition \ref{con:cq_condition} with $\nu_i\geq 0$, as required by dual feasibility.  If Condition~\ref{con:local_condition} holds (LSS), we have $\delta_x  = -\epsilon x \in \mathcal{F}_x$ for $\epsilon\in(0,1)$.  By steps (a)-(d) in eq. \eqref{eq:thm1_rev_ad}, $\!-\lambda^T(\epsilon x) \geq 0$.  By linearity, $\!\mathcal R = -\lambda^Tx \geq 0$ also.
\end{proof}

Observe that the criteria for revenue adequacy in Theorem \ref{thm:revenue_adequacy_general} are related to the GSS property of the constraint set $\s$ as well as the KKT conditions and MFCQ, without dependence on the market mechanism $\m$, the form of the cost function $c(q)$, or any pre-existing commodity allocations $d$. 

\section{Global Property for DC OPF and OGF}  \label{sec:global_star_sets}

We verify the GSS set property for the constraints that define the DC OPF and the steady-state OGF problems. 
Both types of problems involve physical laws that relate the potential difference between two network nodes to the flow on the connecting edge. These physical laws are homogeneous and the potentials are translation invariant, and these two properties imply that these systems satisfy the GSS property.

\vspace{-1ex}
\subsection{DC Optimal Power Flow} \label{sec:dcopf}

Wholesale electricity markets coordinate power transmission and generation facilities using contract network pricing with congestion payments as rental fees for use of capacity rights \cite{hogan1992contract}.  When an FTR holder uses the capacity, the payment of the full rental fee ensures that the marginal opportunity cost of transmission is in fact the actual cost, but average net costs of transmission use are determined by marginal losses only.  If the right holder does not use the capacity, then optimal dispatch is assumed to ensure that the rental payment is at least enough that the net cost to purchase power at the destination is no more than the cost of owning generation and transmission  \cite{caramanis1982optimal}.  The ISO acts as an intermediary, collecting congestion payments from system users and distributing congestion rents to FTR holders.  Revenue adequacy for transmission pricing answers whether total congestion payments in the contract network can cover rental obligations.

The DC optimal power flow (DC OPF) problem is an optimization problem routinely used for operational planning and market clearing in power systems. We suppose that a power system is represented by a set of nodes $\mathcal{N}_p$ connected by lines $\mathcal{E}_p$, with generators $\mathcal{G}_p$. The equations below specify the feasible market and system sets of the DC OPF problem:
\begin{subequations}    \label{eq:dcopf_feasible_set}
\begin{align}
  \m_{DCP} = \bigg\{\quad (\boldsymbol{p}^g,\boldsymbol{p})\mid \quad & \nonumber \\
  \forall i \in \mathcal{N}_p: \quad & p_i = p^g_i - d_i \quad  \label{eq:dcopf_net_injection} \\
  \forall i \in \mathcal{N}_p: \quad & 0\leq p^g_i \leq \overline{p}^g_i \quad  \label{eq:dcopf_gen_lims} \quad \bigg\}, \\
  \s_{DCP} = \bigg\{\quad (\boldsymbol{p},\boldsymbol{\theta},\boldsymbol{p}^l)\mid \quad & \nonumber \\
  \forall i \in \mathcal{N}_p: \quad &\sum_{j \in \partial i} p_{ij}^l = p_i,  \label{eq:dcopf_nodal_conservation}\\
  \forall (ij) \in \mathcal{E}_p: \quad & \theta_i - \theta_j = Z_{ij} p_{ij}^l, \label{eq:dcopf_ohms_law} \\
  \forall (ij) \in \mathcal{E}_p: \quad & -\overline{p}_{ij} \leq p_{ij}^l \leq \overline{p}_{ij} \label{eq:dcopf_branch_flow_limit} \quad \bigg\}.
\end{align}
\end{subequations}
In the above equations, the production of a generator $g\in\mathcal{G}_p$ at a node $i\in\mathcal{N}_p$ is the variable $p_i^g$ with upper bound value $\overline{p}^g_i$, the load at node $i\in\mathcal{N}_p$ is $d_i$, the net power injection at a node $i\in\mathcal{N}_p$ is $p_i$, and the voltage phase angle at a node $i\in\mathcal{N}_p$ is $\theta_i$.  The phase angle difference on a line $(i,j)\in\mathcal{E}_p$ is $\theta_{ij}=\theta_j-\theta_i$.   The variables $p_{ij}^l$ denote power flow on a line $(i,j)\in\mathcal{E}_p$, with line conductance parameters $Z_{ij}$, and $\overline{p}_{ij}$ defines the bound value for the limits on power flow magnitude for each line. We use a shorthand where $\boldsymbol{\theta}$ and $\boldsymbol{p}^l$ are vectors containing $\theta_{ij}$ and $p_{ij}^l$ for all $(i,j)\in\mathcal{E}_p$, respectively, $\boldsymbol{p}$ is a vector containing net nodal power injections $p_i$ for all $i\in\mathcal{N}_p$, and $\boldsymbol{p}^g$ is a vector of generator production levels $p^g_i$ for all $g\in\mathcal{G}_p$.  In equation \eqref{eq:dcopf_nodal_conservation}, $\partial i \subset \mathcal{N}_p$ denotes the set of neighboring nodes to $i\in\mathcal{N}_p$, and we suppose that $p_{ji}^l=-p_{ij}^l$, so that the summation denotes the total power flowing from node $i$ into the network and balancing the net power injection $p_i$ into that node.  The DC OPF uses the so-called DC-approximation, which is a linear approximation of the non-linear power flow physics.  The parameters that are controllable by the operator are the generator production levels, which must in aggregate equal the total system load.  The DC OPF problem minimizes the cost of generation while meeting the demand and the power flow limits, and using a linear cost function with production costs $c_i$ for each generator $i\in\mathcal{G}_p$ leads to the problem \vspace{-2.5ex}
\begin{subequations}
\begin{align}   \label{eq:dcopf}
\min \quad &  \sum_{i\in\mathcal{G}_p} c_i p^g_i \,\, \\\mbox{s.t.}  \quad &   (\boldsymbol{p}^g,\boldsymbol{p}) \in \m_{DCP}, \\ \quad  &  (\boldsymbol{p},\boldsymbol{\theta},\boldsymbol{p}^l) \in \s_{DCP}.
\end{align}
\end{subequations}
The DC OPF problem \eqref{eq:dcopf} is a special case of the general framework in problem \eqref{eq:min_cost_formulation} with the variable correspondence
\begin{align*}
    q \rightarrow \boldsymbol{p}^g, \quad x \rightarrow \boldsymbol{p}, \quad y \rightarrow (\boldsymbol{\theta},\boldsymbol{p}^l)
\end{align*}
and where the constraints 
\eqref{eq:dcopf_nodal_conservation} and  \eqref{eq:dcopf_ohms_law} correspond to the equality constraints $H(x,y)=0$ in \eqref{eq:H_fun_def}, and the constraints \eqref{eq:dcopf_branch_flow_limit} 
correspond to the inequality constraints $G(x,y)\leq0$ in \eqref{eq:G_fun_def}.  The following theorem is a straightforward application of Condition \ref{con:global_condition} and Theorem \ref{thm:revenue_adequacy_general}.

\begin{theorem}  \label{thm:dcopf_star}
The feasible set $\s_{DCP}$ of the DC-OPF satisfies the GSS property. Therefore, any optimal solution to \eqref{eq:dcopf} is revenue adequate.
\end{theorem}
\!\!\!\!\!\!\!\!\!\!\! \begin{proof}
The GSS property for $\s_{DCP}$ is verified from linearity of the inequalities defining the set, i.e., if $(\boldsymbol{p},\boldsymbol{\theta},\boldsymbol{p}^l) \in \s_{DCP}$, then $(s\boldsymbol{p},s\boldsymbol{\theta},s\boldsymbol{p}^l) \in \s_{DCP}$.  By Theorem \ref{thm:revenue_adequacy_general}, $\mathcal R >0$.
\end{proof}

Linear dependence of line flows on bus loads in the DC approximation leads congestion payments by network users to exactly equal capacity rental payments to FTR holders \cite{hogan1992contract}, so that revenue adequacy is inherent.

\vspace{-.5ex}
\subsection{Steady-State Optimal Gas Flow} \label{sec:ssogf}

Natural gas is transported from suppliers to consumers through pipeline networks with potentially complex topology. Energy dissipation caused by turbulent gas flow results in decreasing pressure in the flow direction, so flow is actuated by gas compressors that boost pressure.  In steady-state, it is standard to model gas flow using the Weymouth equation that relates the pressures at the ends of a pipe to the mass flow through the pipe \cite{rios2015optimization}.   Economic gas network flow optimization solutions initially used linear edge capacity models \cite{brooks1981using} or linear approximations of steady-state gas flow equations \cite{oneill1979mathematical}.  Nonlinear optimization has also been used to evaluate the capacity of gas pipeline networks \cite{koch2015evaluating}.  

The steady-state optimal gas flow (OGF) problem was developed to maximize the economic utility for users of a gas pipeline network by maximizing the utilization of its capacity to transport gas from suppliers to consumers subject to nonlinear physics and engineering constraints \cite{rudkevich2017hicss}.  The controllable parameters of the OGF are any loads that can be curtailed, as well as the settings of gas compressors \cite{wu2017adaptive}. We suppose that a gas pipeline network is represented by a set of junctions $\mathcal{N}_g$ with elements $i$ connected by oriented edges in the set $\mathcal{E}_g$, with elements $(i,j)$ that denote pipes.  There is a subset $\mathcal{C}_g\subset\mathcal{E}_g$ of pipes that have a compressor located at the pipe inlet, which boosts the pressure at the pipe inlet with respect to that at the sending node.  The present formulation is limited to consider at most a single variable supply or withdrawal of gas at a junction.  The equations below specify the feasible market and system sets for the OGF problem: \vspace{-1ex}
\begin{subequations}    \label{eq:gas_ss_feasible_set}
\begin{align}
  \m_{OGF} = \bigg\{\,\, (\boldsymbol{q},\boldsymbol{x}) \mid \quad & \nonumber \\
  \forall i \in \mathcal{N}_g: \quad & x_i = q_i - d_i, \label{eq:net_nodal_flow} \\
  \forall i \in \mathcal{N}_g:   \quad &   \underline{q}_i \leq q_i \leq \overline{q}_i  \quad  \bigg\}  \label{eq:gas_trading_bounds} \\
  \vspace{-2.5ex}
  \s_{OGF} = \bigg\{\,\, (\boldsymbol{x}, \boldsymbol{\pi}, \boldsymbol{\phi}, \boldsymbol{\alpha}) \mid \quad & \nonumber \\
  \forall i \in \mathcal{N}_g: \quad & \sum_{j \in \partial_i} \phi_{ij} = x_i \quad  \label{eq:gas_flow_conservation}\\
   \forall (ij) \in \mathcal{E}_g: \quad  &\pi_{ij} - \pi_{ji}  = \beta_{ij} \phi_{ij}|\phi_{ij}| \label{eq:gas_weymouth}\\
   \forall (ij) \in \mathcal{C}_g: \quad  & \pi_i \alpha_{ij} = \pi_{ij},  \, \pi_j = \pi_{ji},  \label{eq:multiplicative_compression} \\
   \forall (ij) \in \mathcal{E}_g\setminus\mathcal{C}_g: \quad  & \pi_i = \pi_{ij},   \, \pi_j = \pi_{ji},\label{eq:no_compression} \\
  \forall (ij) \in \mathcal{E}_g: \quad  & \underline{\pi}_{ij} \leq \pi_{ij} \leq \overline{\pi}_{ij} \label{eq:gas_pressure_bounds_pipes} \\
   \forall (ij) \in \mathcal{E}_g: \quad & 1  \leq \alpha_{ij}  \leq \overline{\alpha}_{ij}  \quad \label{eq:compression_bounds} \\
  \forall i \in \mathcal{N}_g: \quad & \underline{\pi}_i \leq \pi_i \leq \overline{\pi}_i  \label{eq:gas_pressure_bounds_nodal}  \quad  \bigg\}  \vspace{-.5ex}
\end{align}
\end{subequations}
In the above equations, the fixed outflow and optimized inflow of gas from/to the network at a node $i\in\mathcal{N}_g$ is denoted by the variables $d_i$ and $q_i$, respectively, where the latter is positive and is constrained with lower and upper bound values of $\underline{q}_i$ and $\overline{q}_i$, respectively.  The variable $x_i$ then denotes the total inflow of gas into the network at a node $i\in\mathcal{N}_g$.  Gas flow physical state is described by the mass flow $\phi_{ij}$ through each pipe and the square of gas pressure at the start and end of the pipe, denoted by $\pi_{ij}$ and $\pi_{ji}$, respectively, for $(i,j)\in\mathcal{E}_g$.  In the Weymouth equation \eqref{eq:gas_weymouth}, 
the parameter $\beta_{ij}$ is called the \emph{resistance} given for each pipe as $\beta=a^2\lambda L/ (A^2 D)$, where $a$ denotes the local wave speed in the gas, $\lambda$ is the Darcy-Weisbach friction factor, and $L$, $D$, and $A$ denote the length, diameter, and cross-sectional area of the pipe, respectively. In this study, we suppose that gas composition is homogeneous throughout the network, so that $a$ is constant and uniform, while each pipe $(i,j)\in\mathcal{E}_g$ may have various parameters $\lambda_{ij}$, $L_{ij}$, $D_{ij}$, and $A_{ij}$. An auxiliary variable $\pi_i$ is used to represent squared nodal pressure, which is related to the pipe inlet pressure by a ratio $\alpha_{ij}$ that is an optimized variable with upper bound $\overline{\alpha}_{ij}>1$ if a compressor is located at node $i$ and powering flow into pipe $(i,j)\in\mathcal{C}_g$, and $\alpha_{ij}=\overline{\alpha}_{ij}=1$ otherwise.   Gas pressure is constrained in each pipe $(i,j)\in\mathcal{E}_g$, with lower and upper bound values $\underline{\pi}_{ij}$ and $\overline{\pi}_{ij}$, and at each node $i\in\mathcal{N}_g$, with lower and upper bound values $\underline{\pi}_{i}$ and $\overline{\pi}_{i}$. We use a shorthand where $\boldsymbol{q}$, $\boldsymbol{x}$, and $\boldsymbol{\pi}$ are vectors that contain the nodal values $q_i$, $x_i$, and $\pi_i$, respectively, for all $i\in\mathcal{N}_g$, $\boldsymbol{x}$ is a vector that contain $\phi_{ij}$ for all $(i,j)\in\mathcal{E}_g$, and $\boldsymbol{\alpha}$ is a vector that contains $\alpha_{ij}$ for all $(i,j)\in\mathcal{C}_g$. In equation \eqref{eq:gas_flow_conservation}, $\partial i \subset \mathcal{N}_g$ denotes the set of neighboring nodes to $i$, and we suppose that $\phi_{ji}=-\phi_{ij}$, so that the summation denotes the total gas flowing from node $i$ into the network and balancing the net gas injection $x_i$ into that node.  

The OGF problem maximizes economic value generated by the system expressed as receipts from consumers minus the cost of supply (i.e., minimizing the cost of inflows), as well as minimizing the cost of  gas compression to operate the pipeline. Using a linear cost with prices $c_i$ for suppliers or consumers at  nodes $i\in\mathcal{N}_g$, and price functions $c_{ij}$ for the cost to compress gas at the start of pipe $(i,j)$ with ratio $\alpha_{ij}$, leads to the formulation
\begin{subequations}    \label{eq:ss_ogf}
\begin{align}
    \min \quad & \sum_{i \in \mathcal{N}_g} c_iq_i + \sum_{(i,j) \in \mathcal{E}_g} c_{ij}(\alpha_{ij}) \\
    \mbox{s.t.} \quad  & (\boldsymbol{q},\boldsymbol{x}) \in  \m_{OGF}, \\ & (\boldsymbol{x},\boldsymbol{\phi},\boldsymbol{\pi}, \boldsymbol{\alpha}) \in  \s_{OGF}.
\end{align}
\end{subequations}
The OGF can be expressed using the general formulation \eqref{eq:min_cost_formulation} with the variable correspondence
\begin{align*}
    q \rightarrow \boldsymbol{q}, \quad x \rightarrow \boldsymbol{x}, \quad y \rightarrow (\boldsymbol{\pi}, \boldsymbol{\phi}, \boldsymbol{\alpha}), 
\end{align*}
and where constraints 
\eqref{eq:gas_flow_conservation}, \eqref{eq:gas_weymouth}, and \eqref{eq:multiplicative_compression} correspond to the equality constraints $H(x,y)=0$ in \eqref{eq:H_fun_def}, and the constraints 
\eqref{eq:gas_pressure_bounds_pipes}, \eqref{eq:compression_bounds}, and \eqref{eq:gas_pressure_bounds_nodal}
correspond to the inequality constraints $G(x,y)\leq0$ in equation \eqref{eq:G_fun_def}.  Establishing the GSS property for the set $\s_{OGF}$ requires the following assumption.

\begin{assumption}  \label{as:common_pressure_point}
$(\mathcal{N}_g,\mathcal{E}_g)$ is a connected graph and the squared pressure $\pi_c>0$ is allowed at each node $i\in\mathcal{N}_g$:
\begin{align}
    \pi_c \in[\underline{\pi}_i,\overline{\pi}_i] \quad \forall i \in \mathcal{N}_g.
\end{align}
\end{assumption}
\begin{theorem}  \label{thm:gas_ss_star}
If Assumption~\ref{as:common_pressure_point} holds, then the OGF feasible set $\s_{OGF}$ is GSS. Therefore, any locally optimal solution to problem \eqref{eq:ss_ogf} satisfying Condition~\ref{con:cq_condition} is revenue adequate.
\end{theorem}
\begin{proof} 
Let $(\boldsymbol{x},\boldsymbol{\phi},\boldsymbol{\pi},\boldsymbol{\alpha}) \in \s_{OGF}$ be a given feasible point. To establish the theorem, we must show that for any $s \in [0,1]$, there exists $(\phi^s,\pi^s, \alpha^s)$ such that the point $(sx,\phi^s,\pi^s, \alpha^s) \in \s_{OGF}$. Consider the quantities $(\phi^s,\pi^s, \alpha^s)$ defined as a function of $s$ given below:
\begin{subequations}
\begin{align}
    \pi_i^s &= \pi_c + s^2(\pi_i-\pi_c) \quad \forall i \in \mathcal{N}_g, \label{eq:pi_interpol1}\\
    \pi_{ij}^s &= \pi_c + s^2(\pi_{ij}-\pi_c) \quad \forall (ij) \in \mathcal{E}_g, \label{eq:pi_interpol2}\\
    \alpha_{ij}^s &= \frac{\pi_c + s^2(\pi_{ij}-\pi_c)}{\pi_c + s^2(\pi_i-\pi_c)} \quad \forall (ij) \in \mathcal{C}_g. \label{eq:alpha_interpol} \\
    \phi_{ij}^s &= s\phi_{ij} \quad \forall (ij) \in \mathcal{E}_g. \label{eq:phi_interpol}
\end{align}
We show $(sx,\phi^s,\pi^s, \alpha^s) \in \s_{OGF}$.
Equation \eqref{eq:gas_flow_conservation} holds since \vspace{-1ex}
\begin{equation} \label{eq:ssg_flowbal}
    \sum_{j \in \partial i}\phi_{ij}^s = s\sum_{j \in \partial i}\phi_{ij} = s x, \quad \forall \, i \in \mathcal{V}.
\end{equation}
To verify the property for equation \eqref{eq:gas_weymouth}, observe that
\begin{align}
    \pi_{ij}^s - \pi_{ji}^s & = [\pi_c + s^2(\pi_{ij}-\pi_c)] - [\pi_c + s^2(\pi_{ji}-\pi_c)] \nonumber \\
    &  = s^2(\pi_{ij}-\pi_{ji}) \nonumber \\ 
    &  = \phi_{ij}^s|\phi_{ij}^s|.
\end{align}
The compression equation \eqref{eq:multiplicative_compression} holds by the definition of $\alpha_{ij}^s$ in equation \eqref{eq:alpha_interpol}. 
Because $\pi_c,\pi_i \in [\underline{\pi}_i,\overline{\pi}_i]$, we have $\pi_i^s = (1-s^2)\pi_c + s^2 \pi_i \in [\underline{\pi}_i,\overline{\pi}_i]$. A similar argument holds for pipeline pressures thus verifying equations
\eqref{eq:gas_pressure_bounds_nodal} and \eqref{eq:gas_pressure_bounds_pipes}.
To verify the compression ratio bounds, we first observe that if $\alpha_{ij} = 1$ then $\alpha_{ij}^s = 1$ for all $s \in [0,1]$, which trivially satisfies equation \eqref{eq:compression_bounds}. If $\alpha_{ij} \neq 1$, then we compute the ratio \vspace{-2ex}
\begin{align} 
    \frac{\alpha_{ij}^s-1}{\alpha_{ij}-1} = \frac{s^2 \pi_c}{(1-s^2)\pi_c+s^2\pi_i} \in (0,1).
\end{align}
Therefore, $\alpha_{ij} \in [1,\overline{\alpha}_{ij}]$ implies that $\alpha_{ij}^s \in [1,\overline{\alpha}_{ij}]$. Thus $\s_{OGF}$ of the OGF is GSS, and by Theorem \ref{thm:revenue_adequacy_general}, $\mathcal R >0$.
\end{subequations}
\end{proof}
Because OGF constraints $\s_{OGF}$ are GSS, $x=0$ (zero flow) is feasible, and numerical methods that maximize revenue for problem \eqref{eq:ss_ogf} will choose $\mathcal R=0$ rather than a revenue-negative solution.  The reader can verify this using an OGF code \cite{baker2026bling}.

\section{Local Property for AC Optimal Power Flow}  \label{sec:acopf}

The AC OPF is the analog of the DC OPF introduced in Section \ref{sec:dcopf}, with the original, highly accurate non-linear AC power flow equations used instead of the linear DC approximation.  The feasible sets of the AC OPF are \vspace{-1ex}
\begin{subequations}    \label{eq:acopf_feasible_set}
\begin{align}
  \m_{ACP} = \big\{\,\, (\boldsymbol{p}^g,\boldsymbol{q}^g, \boldsymbol{p},\boldsymbol{q}) \mid \quad & \nonumber \\ 
  \forall i \in \mathcal{N}_p: \quad & p_i = p^g_i-p^d_i, \label{eq:total_nodal_injection_p} \\ \forall i \in \mathcal{N}_p: \quad & q_i = q^g_i-q^d_i \label{eq:total_nodal_injection_q} 
  \\  
  \forall i \in \mathcal{N}_p: \quad & 0 \leq p^g_i \leq \overline{p}^g_i, \label{eq:nodal_injection_limits_p} \\ \forall i \in \mathcal{N}_p: \quad &  0 \leq q^g_i \leq \overline{q}^g_i \label{eq:nodal_injection_limits_q} \quad  \big\}
\end{align}
\begin{align}
\s_{ACP}  \!=\! \big\{\!(\boldsymbol{p},\!\boldsymbol{q},\!\boldsymbol{v},\!\boldsymbol{\theta},\!\boldsymbol{p}^l,\!\boldsymbol{q}^l) \! \mid  & \nonumber \\
    \forall i \in \mathcal{N}_p: \,\,\, &\sum_{j \in \partial i} p_{ij}^l = p_i, \label{eq:flow_conservation_p} \\ \forall i \in \mathcal{N}_p: \,\,\, & \sum_{j \in \partial i} q_{ij}^l = q_i, \label{eq:flow_conservation_q}
    \\
    \forall (ij)  \in \mathcal{E}_p: \,\,\,  &p_{ij}^l \!=\! g_{ij}v_i^2  \nonumber \\ & -  g_{ij}v_iv_jcos(\theta_i-\theta_j)  \nonumber \\ & \, - b_{ij}v_iv_jsin(\theta_i-\theta_j) \label{eq:ac_ohms_law_p}
    \\
    \forall (ij)  \in \mathcal{E}_p: \,\,\, &q_{ij}^l \!=\! -\left(b_{ij} + \frac{b_{ij}^{sh}}{2}\right)v_i^2 \nonumber  \\ & \, + b_{ij}v_iv_j cos(\theta_i\!-\!\theta_j) \nonumber \\ & \, - g_{ij}v_iv_j sin(\theta_i\!-\!\theta_j) \label{eq:ac_ohms_law_q}
    \\
     \forall (ij)  \in \mathcal{E}_p: \,\,\, &p_{ij}^2 + q_{ij}^2 \leq \overline{s}_{ij}^2 \label{eq:ac_branch_flow_limit}
     \\
     \forall i  \in \mathcal{N}_p: \,\,\, &\underline{v}_i \leq v_i \leq \overline{v}_i. \label{eq:ac_voltage_bounds}
     \qquad \big\}
\end{align}
\end{subequations}
The AC power flow is defined for a power system represented by a set of nodes $\mathcal{N}_p$ connected by lines $\mathcal{E}_p$, with generators $\mathcal{G}_p$.  In the AC power flow equations, the production of a generator $g\in\mathcal{G}_p$ at a node $i\in\mathcal{N}_p$ is comprised by active and reactive power injections $p_i^g$ and $q_i^g$, respectively, with upper bound values $\overline{p}^g_i$ and $\overline{q}^g_i$.  The active and reactive power loads at node $i\in\mathcal{N}_p$ are $p_i^d$ and $q_i^d$, and the net active and reactive power injections at a node $i\in\mathcal{N}_p$ are $p_i$ and $q_i$. The voltage and voltage phase angle at a node $i\in\mathcal{N}_p$ are $v_i$ and $\theta_i$, respectively.  The variables $p_{ij}^l$ and $q_{ij}^l$ denote active and reactive power flows on a line $(i,j)\in\mathcal{E}_p$, and parameters $g_{ij}$ and $b_{ij}$ denote line conductance and susceptance, respectively. The allowable flow on a line is limited at the bound value $\overline{s}_{ij}$, and nodal voltage is limited by lower and upper bound values  $\underline{v}_{i}$ and $\overline{v}_{i}$.  We consider a shunt susceptance $b_{ij}^{sh}$.  For the AC OPF problem variables, we use a shorthand where $\boldsymbol{\theta}$, $\boldsymbol{p}^l$, and $\boldsymbol{q}^l$ are vectors containing $\theta_{ij}$, $p_{ij}^l$, and $q_{ij}^l$ for all $(i,j)\in\mathcal{E}_p$, respectively. 
We also define vectors $\boldsymbol{p}$, $\boldsymbol{q}$, and $\boldsymbol{v}$ that contain net nodal active and reactive power injections $p_i$ and $q_i$ and node voltages $v_i$ for all $i\in\mathcal{N}_p$. We also denote by $\boldsymbol{p}^g$ and $\boldsymbol{q}^g$ the vectors of active and reactive power production levels $p^g_i$ and $q_i^g$ for generators $g\in\mathcal{G}_p$.  In equations \eqref{eq:flow_conservation_p} and \eqref{eq:flow_conservation_p}, $\partial i \subset \mathcal{N}_p$ denotes the set of neighboring nodes to $i$, and we suppose that $p_{ji}^l=-p_{ij}^l$ and $q_{ji}^l=-q_{ij}^l$, so that the summations in equations \eqref{eq:flow_conservation_p} and \eqref{eq:flow_conservation_p} denote the total active and reactive power flowing from node $i$ into the network and balancing the net injections $p_i$ and $q_i$ into that node, respectively.  The AC-OPF problem is given as
\begin{subequations}    \label{eq:acopf}
\begin{align}
    \min \quad &\sum_{i \in \mathcal{N}_p} c^p_i(p_i^g) + c^q_i(q_i^g) \\
    \mbox{s.t.} \quad &(\boldsymbol{p}^g,\boldsymbol{q}^g, \boldsymbol{p},\boldsymbol{q}) \in \m_{ACP}. \\ \quad &(\boldsymbol{p},\!\boldsymbol{q},\!\boldsymbol{v},\!\boldsymbol{\theta},\!\boldsymbol{p}^l,\!\boldsymbol{q}^l) \in \s_{ACP}.
\end{align}
\end{subequations}
The AC OPF problem can be cast in the format of problem \eqref{eq:min_cost_formulation} using the variable correspondence \vspace{-.5ex}
\begin{align*}
    q \rightarrow (\boldsymbol{p}^g,\boldsymbol{q}^g), \quad x \rightarrow (\boldsymbol{p},\boldsymbol{q}), \quad y \rightarrow (\boldsymbol{p},\boldsymbol{q},\boldsymbol{v},\boldsymbol{\theta},\boldsymbol{p}^l,\boldsymbol{q}^l),
\end{align*}
and where the constraints  \eqref{eq:flow_conservation_p},  \eqref{eq:flow_conservation_p}, \eqref{eq:ac_ohms_law_p},  and  \eqref{eq:ac_ohms_law_q} correspond to the equality constraints $H(x,y)=0$ in \eqref{eq:H_fun_def}, and the constraints \eqref{eq:ac_branch_flow_limit} and \eqref{eq:ac_voltage_bounds} correspond to the inequality constraints $G(x,y)\leq0$ in \eqref{eq:G_fun_def}.  
The next theorem establishes revenue adequacy of local optima of the AC-OPF problem.

\begin{theorem} \label{thm:ac_shrinkage}
Any feasible point in $\s_{ACP}$ where the voltage lower bounds in equation \eqref{eq:ac_voltage_bounds} are not binding satisfies the LSS property Condition~\ref{con:local_condition}. If the point is also locally optimal and Condition~\ref{con:cq_condition} is satisfied, then it is revenue adequate.
\end{theorem}
\begin{proof} 
By homogeneity of the power flow equations, given a point $(\boldsymbol{p},\boldsymbol{q},\boldsymbol{v},\boldsymbol{\theta},\boldsymbol{p}^l,\boldsymbol{q}^l) \in \s_{ACP}$, then the point $(s\boldsymbol{p},s\boldsymbol{q},\sqrt{s} \boldsymbol{v},\boldsymbol{\theta},s\boldsymbol{p}^l, s\boldsymbol{q}^l)$ satisfies \eqref{eq:flow_conservation_p}-\eqref{eq:ac_branch_flow_limit}. Further, if the voltage lower bounds are not tight, i.e., $v_i > \underline{v}_i$ for all $i \in \mathcal{V}$, then there exists $\epsilon\in(0,1)$ such that $ \underline{v}_i \leq \sqrt{s}v_i \leq \overline{v}_i$ for all $1-\epsilon \leq s \leq 1$. Therefore the direction $(-\boldsymbol{p},-\boldsymbol{q})$ is feasible and $(\boldsymbol{p},\boldsymbol{q},\boldsymbol{v},\boldsymbol{\theta},\boldsymbol{p}^l,\boldsymbol{q}^l)$ satisfies the LSS property. Therefore $\s_{ACP}$ is LSS, and by Theorem \ref{thm:revenue_adequacy_general}, $\mathcal R >0$.
\end{proof}

\section{Conclusions}       \label{sec:conc}
\vspace{-.5ex}

We present concise and general conditions for revenue adequacy in markets for energy transport, and verify the assumptions for DC power flow, AC power flow, and steady-state gas pipeline flow. 
Our general theory of revenue adequacy for network flows enables policy decisions for relieving congestion as well as FTR market mechanisms that do not require strict convexity of feasible sets. The adoption of contemporary energy market designs such as FTR mechanisms has been suggested to require policies to accommodate congestion revenue shortfalls \cite{lesieutre2005convexity}.  Our results provide a constructive technique to verify that FTR mechanisms that rely on AC power flow modeling can be revenue adequate under certain conditions.  \vspace{1ex}
 
\bibliographystyle{IEEEtran}
\linespread{.96}
\bibliography{references}

\end{document}